\documentclass[a4paper,final]{amsart}

\usepackage{showkeys}
\usepackage{fullpage}
\usepackage{mathtools, amssymb, amsthm, amsfonts, mathrsfs}
\usepackage{enumitem}

\usepackage{color, graphicx, here}

\usepackage{tikz}
\usetikzlibrary{intersections, calc, arrows, cd}
\usepackage[colorlinks,citecolor=darkgreen,linkcolor=red]{hyperref}
\definecolor{darkgreen}{rgb}{0.0, 0.6, 0.0}

\numberwithin{equation}{section}
\numberwithin{figure}{section}
\allowdisplaybreaks

\setenumerate{label={\upshape (\arabic*)}, nosep}
\setitemize{nosep}
\setdescription{nosep}
\newlist{enua}{enumerate}{1}
\setlist*[enua]{label={\upshape (\arabic*)}, nosep}
\newlist{enur}{enumerate}{1}
\setlist*[enur]{label={\upshape (\roman*)}, nosep}

\usepackage{cleveref}
\crefname{thm}{Theorem}{Theorems}
\crefname{thma}{Theorem}{Theorems}
\crefname{prp}{Proposition}{Propositions}
\crefname{lem}{Lemma}{Lemmas}
\crefname{cor}{Corollary}{Corollaries}
\crefname{dfn}{Definition}{Definitions}
\crefname{rmk}{Remark}{Remarks}
\crefname{fct}{Fact}{Facts}
\crefname{ex}{Example}{Examples}
\crefname{ques}{Question}{Questions}

\theoremstyle{plain}
\newtheorem{thm}{Theorem}[section]
\newtheorem{prp}[thm]{Proposition}
\newtheorem{lem}[thm]{Lemma}
\newtheorem{cor}[thm]{Corollary}
\newtheorem{fct}[thm]{Fact}

\newtheorem{thma}{Theorem}

\theoremstyle{definition}
\newtheorem{dfn}[thm]{Definition}

\newtheorem{rmk}[thm]{Remark}
\newtheorem{ex}[thm]{Example}

\newcommand{\ass}{\mathrm{ass}}
\DeclareMathOperator{\Sub}{\mathsf{Sub}}
\DeclareMathOperator{\Fac}{\mathsf{Fac}}
\DeclareMathOperator{\Filt}{\mathsf{Filt}}
\DeclareMathOperator{\upper}{\mathsf{upper}} 

\newcommand{\T}{\mathsf{T}}
\newcommand{\F}{\mathsf{F}}

\DeclareMathOperator{\ie}{\mathsf{ie}}
\DeclareMathOperator{\ice}{\mathsf{ice}}
\DeclareMathOperator{\ike}{\mathsf{ike}}
\DeclareMathOperator{\ke}{\mathsf{ke}}
\DeclareMathOperator{\ce}{\mathsf{ce}}
\DeclareMathOperator{\torf}{\mathsf{torf}}
\DeclareMathOperator{\tors}{\mathsf{tors}}
\DeclareMathOperator{\serre}{\mathsf{serre}}
\DeclareMathOperator{\wide}{\mathsf{wide}}
\DeclareMathOperator{\domres}{\mathsf{dom}\text{.}\mathsf{resol}}

\newcommand{\imply}{\Rightarrow}
\newcommand{\equi}{\Leftrightarrow}

\newcommand{\Equi}{\quad \Longleftrightarrow \quad}

\newcommand{\ol}{\overline}

\newcommand{\xr}[1]{\xrightarrow{\, #1 \, }}
\newcommand{\inj}{\hookrightarrow}
\newcommand{\surj}{\twoheadrightarrow}

\newcommand{\simto}{\xr{\sim}}
\newcommand{\isoto}{\xr{\iso}}

\newcommand{\iso}{\cong}

\newcommand{\bbD}{\mathbb{D}}

\newcommand{\catA}{\mathcal{A}}
\newcommand{\catB}{\mathcal{B}}
\newcommand{\catC}{\mathcal{C}}

\newcommand{\catF}{\mathcal{F}}

\newcommand{\catS}{\mathcal{S}}
\newcommand{\catT}{\mathcal{T}}

\newcommand{\catX}{\mathcal{X}}
\newcommand{\catY}{\mathcal{Y}}

\DeclareMathOperator{\catmod}{\mathsf{mod}}

\DeclareMathOperator{\coh}{\mathsf{coh}}

\DeclareMathOperator{\cm}{\mathsf{cm}}

\DeclareMathOperator{\Hom}{Hom}
\DeclareMathOperator{\End}{End}

\DeclareMathOperator{\Ima}{Im}
\DeclareMathOperator{\Ker}{Ker}
\DeclareMathOperator{\Cok}{Cok}
\DeclareMathOperator{\ZZ}{Z} 

\newcommand{\pb}{\arrow[rd,"{\mathrm{PB}}",phantom]}

\newcommand{\op}{\text{op}}

\DeclareMathOperator{\Spec}{Spec}
\DeclareMathOperator{\Max}{Max}
\DeclareMathOperator{\Min}{Min}
\DeclareMathOperator{\Ass}{Ass}

\DeclareMathOperator{\htt}{ht}
\DeclareMathOperator{\depth}{depth}
\DeclareMathOperator{\grade}{grade} 

\DeclareMathOperator{\ann}{ann}

\newcommand{\mm}{\mathfrak{m}}

\newcommand{\pp}{\mathfrak{p}}
\newcommand{\qq}{\mathfrak{q}}

\newcommand{\kk}{\kappa}

\DeclareMathOperator{\Supp}{Supp}
\newcommand{\shO}{\mathscr{O}}

\newcommand{\LL}{\Lambda}

\newcommand{\pow}{\mathsf{pow}}

\newcommand{\arr}[1]{\arrow[{#1}]}
\newenvironment{bsmatrix}{\left[\begin{smallmatrix}}{\end{smallmatrix}\right]}

\begin{document}
\title{When are KE-closed subcategories torsion-free classes?}

\author{Toshinori Kobayashi}
\address{School of Science and Technology, Meiji University, 1-1-1 Higashi-Mita, Tama-ku, Kawasaki-shi, Kanagawa 214-8571, Japan}
\email{tkobayashi@meiji.ac.jp}

\author{Shunya Saito}
\address{Graduate School of Mathematics, Nagoya University, Chikusa-ku, Nagoya. 464-8602, Japan}
\email{m19018i@math.nagoya-u.ac.jp}


\begin{abstract}
Let $R$ be a commutative noetherian ring and 
denote by $\catmod R$ the category of finitely generated $R$-modules.
In this paper,
we study KE-closed subcategories of $\catmod R$,
that is, additive subcategories closed under kernels and extensions.
We first give a characterization of KE-closed subcategories: 
\emph{a KE-closed subcategory is a torsion-free class in a torsion-free class}.
As an immediate application of the dual statement, 
we give a conceptual proof of Stanley-Wang's result about narrow subcategories.
Next, we classify the KE-closed subcategories of $\catmod R$ 
when $\dim R \le 1$ and when $R$ is a two-dimensional normal domain.
More precisely, in the former case, we prove that
KE-closed subcategories coincide with torsion-free classes in $\catmod R$.
Moreover, this condition implies $\dim R \le 1$ 
when $R$ is a homomorphic image of a Cohen-Macaulay ring
(e.g.\ a finitely generated algebra over a regular ring).
Thus, we give a complete answer for the title.
\end{abstract}

\maketitle
\tableofcontents

\section{Introduction}

Let $R$ be a commutative noetherian ring and 
denote by $\catmod R$ the category of finitely generated $R$-modules.
Homological properties of $\catmod R$ are closely related to
ring-theoretic properties of $R$.
For example,
Auslander-Buchsbaum-Serre's landmark theorem \cite{AB1,AB2,Ser} states that 
the global dimension of $\catmod R$ is finite if and only if $R$ is regular
when $R$ is local.
Moreover, in this case,
the global dimension of $\catmod R$ coincides with the Krull dimension of $R$.
As a consequence of this theorem, they proved that
a localization of a regular local ring is again regular.
The proof without Auslander-Buchsbaum-Serre's theorem is still unknown.
Their works inspired the study of homological properties of $\catmod R$
in commutative ring theory.

One of the principal approaches to studying homological properties of $\catmod R$
is classifying its subcategories closed under some operations. 
It has been studied by several authors so far.
One of the most classical and influential results is 
Gabriel's classification of Serre subcategories \cite{Gab}.
A \emph{Serre subcategory} is an additive subcategory 
closed under subobjects, quotients and extensions.
He established an explicit one-to-one correspondence between
the Serre subcategories of $\catmod R$ and the upper sets of the prime spectrum $\Spec R$.
The next most important result is 
Takahashi's classification of torsion-free classes  \cite{Takahashi}.
A \emph {torsion-free class} is an additive subcategory 
closed under subobjects and extensions.
He constructed an explicit bijection between 
the torsion-free classes of $\catmod R$ and the subsets of $\Spec R$.
The classification of other classes of subcategories is often reduced to these two classifications.
In fact,
most classes of subcategories in $\catmod R$ coincide with 
either Serre subcategories or torsion-free classes.
The following diagram summarizes the implications of various subcategories of an abelian category
(see \cref{dfn:several subcat}):
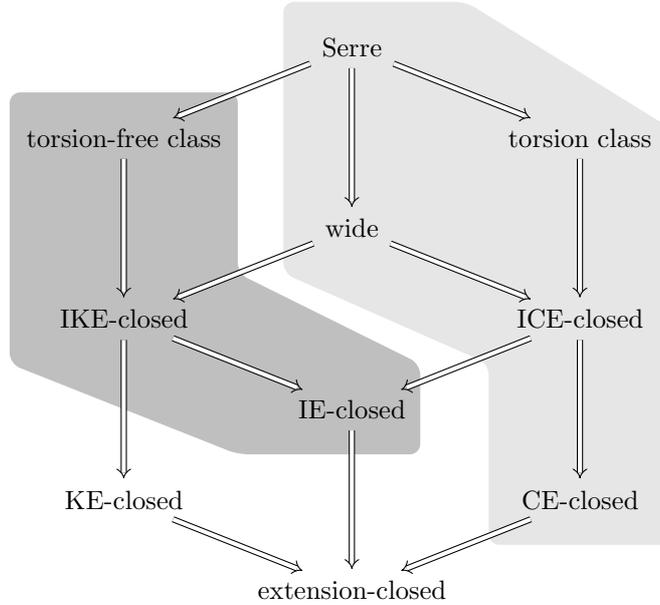
\begin{figure}[H]
\begin{tikzpicture}[
arrow/.style={-Implies,double equal sign distance,shorten >=1pt,shorten <=1pt},
xscale=3, yscale=1.2
]

\fill[gray!20, rounded corners]
(-0.3, 0.5) -- (-0.3, -2.5) -- (0.6, -3.5) -- (0.6, -5.5) -- (1.4, -5.5) --
(1.4, -0.9) -- (0.5, 0.5) -- cycle;

\fill[gray!50, rounded corners]
(-1.5, -0.5) -- (-1.5, -3.5) -- (-0.5, -4.5) --
(0.3, -4.5) -- (0.3, -3.5) -- (-0.5, -2.5) -- (-0.5, -0.5) -- cycle;

\node (Serre) at (0, 0) {Serre};
\node (torsion-free) at (-1, -1) {torsion-free class};
\node (torsion) at (1, -1) {torsion class};
\node (wide) at (0, -2) {wide};
\node (IKE) at (-1, -3) {IKE-closed};
\node (ICE) at (1, -3) {ICE-closed};
\node (IE) at (0, -4) {IE-closed};
\node (KE) at (-1, -5) {KE-closed};
\node (CE) at (1, -5) {CE-closed};
\node (E) at (0, -6) {extension-closed};

\draw[arrow] (Serre) -- (torsion);
\draw[arrow] (Serre) -- (torsion-free);
\draw[arrow] (Serre) -- (wide);
\draw[arrow] (torsion-free) -- (IKE);
\draw[arrow] (torsion) -- (ICE);
\draw[arrow] (wide) -- (IKE);
\draw[arrow] (wide) -- (ICE);
\draw[arrow] (ICE) -- (IE);
\draw[arrow] (IKE) -- (IE);
\draw[arrow] (ICE) -- (CE);
\draw[arrow] (IKE) -- (KE);
\draw[arrow] (CE) -- (E);
\draw[arrow] (KE) -- (E);
\draw[arrow] (IE) -- (E);
\end{tikzpicture}
\end{figure}
The contributions of the various authors \cite{Takahashi,SW,IMST,Eno} have shown that
the classes of subcategories in each of the two gray regions in the above diagram 
coincide in $\catmod R$.
For example, Stanley-Wang \cite{SW} proved that
CE-closed subcategories (a.k.a.\ narrow subcategories) coincide with Serre subcategories in $\catmod R$.
See \cite{IK,Sai2} for extensions of these results to noncommutative rings and schemes respectively.

The present paper aims to study KE-closed subcategories of $\catmod R$,
that is, additive subcategories closed under kernels and extensions.
It is natural to ask
whether KE-closed subcategories and torsion-free classes coincide in $\catmod R$.
However, there is an easy counter example:
for a two-dimensional Cohen-Macaulay ring $R$,
the category $\cm R$ of maximal Cohen-Macaulay $R$-modules
is a KE-closed subcategory but not a torsion-free class of $\catmod R$.
The next most natural question is 
when KE-closed subcategories and torsion-free classes coincide in $\catmod R$.
The following result gives a complete answer for this question, under some mild assumptions.
\begin{thma}[{$=$ \cref{prp:ke=torf imply dim=1,KE=torf in 1-dim}}]\label{thma:KE=torf}
Consider the following conditions for a commutative noetherian ring $R$.
\begin{enur}
\item
$\dim R \le 1$.
\item
KE-closed subcategories and torsion-free classes coincide in $\catmod R$.
\end{enur}
Then the implication ``$\mathrm{(i)}\imply\mathrm{(ii)}$'' holds.
Moreover,
the equivalence ``$\mathrm{(i)}\equi\mathrm{(ii)}$'' holds
when $R$ is a homomorphic image of a Cohen-Macaulay ring
(e.g.\ a finitely generated algebra over a regular ring).
\end{thma}
\cref{thma:KE=torf} characterizes the ring-theoretic property $\dim R \le 1$
by the homological property of $\catmod R$.
Since the torsion-free classes of $\catmod R$ were classified by Takahashi
(cf.\ \cref{fct:Takahashi}), 
\cref{thma:KE=torf} also gives a classification of KE-closed subcategories of $\catmod R$
when $\dim R \le 1$.
We also classify the KE-closed subcategories of $\catmod R$
when $R$ is a two-dimensional normal domain 
(see \cref{prp:classify KE when two-dim normal local domain,rmk:classification of dom resol}).
Especially, we obtain the following classification result:
\begin{thma}[{$=$ \cref{prp:classify KE when two-dim normal local domain}}]\label{thmb:KE=torf+cm}
Let $R$ be a two-dimensional noetherian normal local domain.
Then a KE-closed subcategory of $\catmod R$ is either of the following:
\begin{enur}
\item
a torsion-free class of $\catmod R$,
\item
the category $\cm R$ of maximal Cohen-Macaulay $R$-modules.
\end{enur}
\end{thma}

A key ingredient of the proofs of \cref{thma:KE=torf,thmb:KE=torf+cm} is the notion of dominant resolving subcategories.
This intriguing class of subcategories was introduced by Dao and Takahashi \cite{DT}.
For the precise definition, see \cref{dfn:dom res}.
We note that many basic facts on dominant resolving subcategories, including their classification, are established in \cite{DT, Takahashi2}.
We reveal how dominant resolving subcategories relate to KE-closed subcategories (see \cref{KE-closed res}).

When $R$ is a Cohen-Macaulay ring with a canonical module,
there is another proof of \cref{thma:KE=torf} which is much simpler than the general case
(see \cref{prp:KE=torf when 1-dim CM with can mod}).
The proof is based on the following characterization of KE-closed subcategories
in terms of a sequence of subcategories.
\begin{thma}[{$=$\cref{prp:KE=torf in torf,prp:CE=tors in tors}}]\label{thma:KE=torf in torf}
Let $\catX$ be a subcategory of an abelian category $\catA$.
\begin{enua}
\item
The following are equivalent.
\begin{itemize}
\item
$\catX$ is KE-closed in $\catA$.
\item
There exists a torsion-free class $\catF$ of $\catA$
such that $\catX$ is a torsion-free class of $\catF$
in the sense of exact categories (see \cref{dfn:subcat in ex cat}).
\end{itemize}

\item
The following are equivalent.
\begin{itemize}
\item
$\catX$ is CE-closed in $\catA$.
\item
There exists a torsion class $\catT$ of $\catA$
such that $\catX$ is a torsion-free class of $\catT$
in the sense of exact categories.
\end{itemize}
\end{enua}
\end{thma}

As another consequence of \cref{thma:KE=torf in torf},
we can prove Stanley-Wang's result on the level of the theory of abelian categories.
\begin{thma}[{$=$\cref{prp:subcat when Serre=tors}}]
Let $\catA$ be an abelian category.
If Serre subcategories and torsion classes coincide in $\catA$,
then Serre subcategories and CE-closed subcategories also coincide in $\catA$.
\end{thma}

Finally, combining \cref{thma:KE=torf,thma:KE=torf in torf},
we obtain the following classifications of torsion(-free) classes of 
a torsion-free class of $\catmod R$.
\begin{thma}[{$=$\cref{prp:tors=Serre in torf,prp:classify torf in torf}}]\label{thma:classify subcat in torf}
Let $R$ be a commutative noetherian ring.
Let $\Phi$ be a subset of $\Spec R$
and denote by $\catmod^{\ass}_{\Phi} R$ the torsion-free class corresponding to $\Phi$ 
by Takahashi's classification (see \cref{fct:Takahashi}).
\begin{enua}
\item
Serre subcategories and torsion classes coincide in $\catmod^{\ass}_{\Phi} R$.
\item
If $\dim R \le 1$, there is a bijection between the following:
\begin{itemize}
\item 
The set of torsion-free classes of $\catmod^{\ass}_{\Phi} R$.
\item
The power set of $\Phi$.
\end{itemize}
\end{enua}
\end{thma}
Since the Serre subcategories of $\catmod^{\ass}_{\Phi} R$ were classified 
via the upper set of $\Phi$ in \cite[Corollary 5.10]{Sai2} 
(see also \cref{fct:classify Serre in torf}),
we obtain the complete classifications of Serre subcategories and torsion(-free) classes
of the torsion-free class $\catmod^{\ass}_{\Phi} R$,
which are summarized as follows:
\[
\begin{tikzcd}[row sep=8pt,column sep=40pt,ampersand replacement=\&]
\{\text{torsion classes in $\catmod^{\ass}_{\Phi} R$}\} \arr{d,equal} \& \\
\{\text{Serre subcategories in $\catmod^{\ass}_{\Phi} R$}\} \arr{r,"\iso",leftrightarrow} \arr{d,"\subseteq",phantom,sloped} \& \{\text{upper sets of $\Phi$}\} \arr{d,"\subseteq",phantom,sloped} \\
\{\text{torsionfree classes in $\catmod^{\ass}_{\Phi} R$}\} \arr{r,"\iso","\dim R \le 1"',leftrightarrow} \& \{\text{subsets of $\Phi$}\}.
\end{tikzcd}
\]
These results extend the classical classifications of various subcategories 
of the abelian category $\catmod R$ by \cite{Gab,Takahashi,SW} 
to the exact category $\catmod^{\ass}_{\Phi} R$.

\medskip
\noindent
{\bf Conventions.}
For a category $\catC$,
we denote by $\Hom_{\catC}(M,N)$ the set of morphisms between objects $M$ and $N$ in $\catC$.
In this paper,
we suppose that all subcategories are full subcategories closed under isomorphisms.
Thus, we often identify the subcategories with the subsets of the set of isomorphism classes
of objects in $\catC$.
A subcategory of an additive category is said to be \emph{additive}
if it is closed under finite direct sums.
In particular, an additive subcategory always contains zero objects.


Let $R$ be a commutative ring.
We denote by $\Spec R$ the set of prime ideals of $R$
and $\Min R$ (resp. $\Max R$) the set of minimal prime (resp.\ maximal) ideals of $R$.
For any $\pp \in \Spec R$,
we denote by $\kk(\pp):=R_{\pp}/{\pp R_{\pp}}$ the residue field of $R$ at $\pp$.
We denote by $\catmod R$ 
the category of finitely generated (left) $R$-modules.
For $M\in \catmod R$, we denote by $\Ass M$ the set of associated prime ideals of $M$
and $\Supp M$ the support of $M$.

\medskip
\noindent
{\bf Acknowledgement.}
The second author would like to thank Osamu Iyama for giving him 
the opportunity to study KE-closed subcategories.
He would also like to thank Haruhisa Enomoto for sharing his observation
that motivated the content of Section \ref{s:KE=torf in torf}.
He is very grateful to Ryo Takahashi for kindly answering the author's questions
and for his valuable comments.

The first author is supported by JSPS Grant-in-Aid for JSPS Fellows 21J00567.
The second author is supported by JSPS KAKENHI Grant Number JP21J21767.

\section{Preliminaries: several subcategories of an abelian category}\label{s:pre}
In this section,
we briefly  recall some properties of subcategories of an abelian category.
\begin{dfn}\label{dfn:several subcat}
Let $\catA$ be an abelian category and $\catX$ its additive subcategory.
\begin{enua}
\item
$\catX$ is said to be \emph{closed under extensions} (or \emph{extension-closed})
if for any exact sequence $0 \to A \to B \to C \to 0$,
we have that $A,C \in \catX$ implies $B\in \catX$.
\item
$\catX$ is said to be \emph{closed under images} (resp.\ \emph{kernels}, resp.\ \emph{cokernels})
if for any morphism $f \colon X \to Y$ in $\catA$ with $X,Y \in \catX$,
we have that $\Ima f \in \catX$ (resp.\ $\Ker f \in \catX$, resp.\ $\Cok f \in \catX$).
\item
$\catX$ is said to be \emph{closed under subobjects} (resp.\ \emph{quotients})
if for any injection $A \inj X$ (resp.\ surjection $X \surj A$) in $\catA$ such that $X\in \catX$,
we have that $A \in \catX$.
\item
$\catX$ is called a \emph{Serre subcategory}
if it is closed under subobjects, quotients and extensions.
\item
$\catX$ is called a \emph{torsion-free class} (resp.\ \emph{torsion class})
if it is closed under extensions and subobjects (resp.\ extensions and quotients).
\item
$\catX$ is said to be \emph{wide}
if it is closed under kernels, cokernels and extensions.
\item
$\catX$ is said to be \emph{IKE-closed}
if it is closed under images, kernels and extensions.
\item
$\catX$ is said to be \emph{ICE-closed}
if it is closed under images, cokernels and extensions.
\item
$\catX$ is said to be \emph{IE-closed}
if it is closed under images and extensions.
\item
$\catX$ is said to be \emph{KE-closed} 
if it is closed under kernels and extensions.
\item
$\catX$ is said to be \emph{CE-closed} (or \emph{narrow})
if it is closed under cokernels and extensions.
\item
We denote by $\serre \catA$, $\torf \catA$, $\tors \catA$, $\wide \catA$, $\ike \catA$,
$\ice \catA$, $\ie \catA$, $\ke \catA$ and $\ce \catA$,
the set of Serre subcategories, torsion-free classes, torsion classes, wide subcategories,
IKE-closed subcategories, ICE-closed subcategories, IE-closed subcategories,
KE-closed subcategories and CE-closed subcategories of $\catA$, respectively.
\end{enua}
\end{dfn}
It is easy to see that the following implications hold:
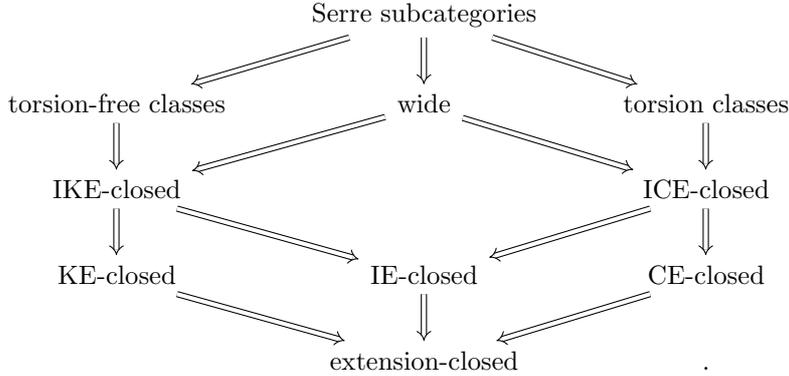
\[
\begin{tikzcd}
&\text{Serre subcategories} \arr{ld,Rightarrow} \arr{d,Rightarrow} \arr{rd,Rightarrow}& \\
\text{torsion-free classes} \arr{d,Rightarrow} & \text{wide} \arr{ld,Rightarrow} \arr{rd,Rightarrow} & \text{torsion classes} \arr{d,Rightarrow}\\
\text{IKE-closed} \arr{d,Rightarrow} \arr{rd,Rightarrow} && \text{ICE-closed} \arr{ld,Rightarrow} \arr{d,Rightarrow}\\
\text{KE-closed} \arr{rd,Rightarrow}& \text{IE-closed} \arr{d,Rightarrow} & \text{CE-closed} \arr{ld,Rightarrow}\\
& \text{extension-closed} & .
\end{tikzcd}
\]

\begin{rmk}\label{rmk:closed under DS}
Let $\catA$ be an abelian category and $\catX$ its additive subcategory.
If $\catX$ is closed under images (resp.\ kernels, resp.\ cokernels),
then it is closed under direct summands.
That is, for any $X, Y \in \catA$, we have that 
$X\oplus Y \in \catX$ implies both $X$ and $Y$ belong to $\catX$.
Indeed, consider the following morphisms:
\[
e_{11}:=\begin{bsmatrix} 1 & 0\\ 0 & 0 \end{bsmatrix}, e_{22}:=\begin{bsmatrix} 0 & 0\\ 0 & 1 \end{bsmatrix} \colon X \oplus Y \to X \oplus Y.
\]
Then we have
$X=\Ker(e_{22})=\Ima(e_{11})=\Cok(e_{22})$ and $Y=\Ker(e_{11})=\Ima(e_{22})=\Cok(e_{11})$.
Thus, we have $X, Y \in \catX$ in any case.
\end{rmk}

We often use the following notation about subcategories of an abelian category.
\begin{dfn}\label{dfn:notation subcat}
Let $\catX$ and $\catY$ be subcategories of an abelian category $\catA$.
\begin{enua}
\item
We denote by $\Sub \catX$ the subcategory of $\catA$
consisting of subobjects of objects of $\catX$.
\item
We denote by $\Fac \catX$ the subcategory of $\catA$
consisting of quotients of objects of $\catX$.
\item
We denote by $\catX * \catY$ the subcategory of $\catA$
consisting of $M \in \catA$ such that there is an exact sequence
$0\to X \to M \to Y \to 0$ of $\catA$ for some $X\in\catX$ and $Y\in \catY$.
\item
We define the subcategory $\Filt \catX$ of $\catA$ by
\[
\Filt \catX :=\bigcup_{n\ge 0} \catX^{* n},
\]
where $\catX^{* (n+1)}:=\catX^{* n} * \catX$ and $\catX^{* 0}= \{0\}$.
\item
We define the subcategories of $\catA$ by
$\F(\catX):=\Filt(\Sub \catX)$ and $\T(\catX):=\Filt(\Fac \catX)$.
The former is called the \emph{torsion-free closure} of $\catX$
and the latter is called the \emph{torsion closure} of $\catX$.
\end{enua}
\end{dfn}

\begin{fct}[{see, for example, \cite[Lemma 2.5 and 2.6]{Eno}}]
The following hold for a subcategory $\catX$ of an abelian category $\catA$.
\begin{enua}
\item
$\F(\catX)$ is the smallest torsion-free class of $\catA$ containing $\catX$.
\item
$\T(\catX)$ is the smallest torsion class of $\catA$ containing $\catX$.
\end{enua}
\end{fct}

Next, we introduce the notion of Serre subcategories and torsion(-free) classes
of an extension-closed subcategory,
which is one of the main topics of this paper.
\begin{dfn}\label{dfn:subcat in ex cat}
Let $\catX$ be an extension-closed subcategory of an abelian category $\catA$.
\begin{enua}
\item
A \emph{conflation} of $\catX$ is 
an exact sequence $0 \to A \to B \to C \to 0$ of $\catA$
such that $A$, $B$ and $C$ belong to $\catX$.
\item
An additive subcategory $\catS$ of $\catX$ is said to be \emph{closed under conflations}
if for any conflation $0\to X \to Y \to Z \to 0$ in $\catX$,
we have that $X,Z \in \catS$ implies $Y\in \catS$.
\item
An additive subcategory $\catS$ of $\catX$ is said to be \emph{closed under admissible subobjects}
(resp.\ \emph{admissible quotients})
if for any conflation $0\to X \to Y \to Z \to 0$ in $\catX$,
we have that $Y\in\catS$ implies $X\in\catS$ (resp.\ $Z\in \catS$).
\item
A \emph{Serre subcategory of $\catX$} is an additive subcategory of $\catX$
closed under conflations, admissible subobjects, and admissible quotients.
\item
A \emph{torsion-free class of $\catX$} is an additive subcategory of $\catX$
closed under conflations and admissible subobjects.
\item
A \emph{torsion class of $\catX$} is an additive subcategory of $\catX$
closed under conflations and admissible quotients.
\item
We denote by $\serre \catX$, $\torf \catX$ and $\tors \catX$,
the set of Serre subcategories, torsion-free classes and torsion classes of $\catX$, respectively.
\end{enua}
\end{dfn}

We also introduce an appropriate duality on an extension-closed subcategory.
\begin{dfn}\label{dfn:ex dual}
Let $\catX$ be an extension-closed subcategory of an abelian category $\catA$.
An \emph{exact duality} on $\catX$ is an additive equivalence $F \colon \catX^{\op} \to \catX$
with a quasi-inverse $G$ such that both $F$ and $G$ preserve conflations of $\catX$.
That is, for any conflation $0 \to X \to Y \to Z \to 0$ in $\catX$,
the sequences $0 \to FX \to FY \to FZ \to 0$ and $0 \to GX \to GY \to GZ \to 0$
are also conflations of $\catX$.
\end{dfn}

\section{Characterizations of several subcategories via their torsion(-free) closures}\label{s:KE=torf in torf}
In this section,
we characterize several subcategories of an abelian category
in terms of a sequence of extension-closed subcategories.
These results play a central role throughout this paper.

Let us begin with a study of KE-closed subcategories via their torsion-free closures. 
The following proposition means that
\emph{a KE-closed subcategory is a torsion-free class in a torsion-free class}.
\begin{prp}\label{prp:KE=torf in torf}
The following are equivalent for a subcategory $\catX$ of an abelian category $\catA$.
\begin{enur}
\item
$\catX$ is a KE-closed subcategory of $\catA$.
\item
$\catX$ is a torsion-free class of $\F(\catX)$.
\item
There exists a torsion-free class $\catF$ of $\catA$
such that $\catX$ is a torsion-free class of $\catF$.
\end{enur}
\end{prp}
\begin{proof}
(i)$\imply$(ii):
It is easy to see that $\catX$ is closed under conflations in $\F(\catX)$.
We prove the following by the induction on $n$:
\begin{itemize}
\item 
For any conflation $0\to A \to X \to B \to 0$ in $\F(\catX)$ such that $X\in \catX$
and $B\in (\Sub\catX)^{* n}$,
we have $A\in \catX$.
\end{itemize}
Suppose that $n=1$, that is, $B\in \Sub \catX$.
There is an injection $B \inj X'$ to an object $X'$ of $\catX$.
Then we have $A=\Ker\left(X \surj B \inj X'\right)\in\catX$ 
since $\catX$ is closed under kernels in $\catA$.
Next, we suppose that $n>1$.
Then there exists an exact sequence $0 \to M \xr{i} B \to N \to 0$ in $\catA$
such that $M \in (\Sub\catX)^{*(n-1)}$ and $N\in \Sub \catX$.
Pulling back the exact sequence $0\to A \to X \to B \to 0$ by $i$,
we have the following commutative diagram:
\[
\begin{tikzcd}
&&0 \arr{d} & 0 \arr{d} & \\
0 \arr{r} & A \arr{r} \arr{d,equal} & Y \arr{r} \arr{d} \pb & M \arr{r} \arr{d,"i"} &0\\
0 \arr{r} & A \arr{r} & X \arr{r} \arr{d} & B \arr{r} \arr{d} &0\\
&& N \arr{d} \arr{r,equal} & N \arr{d}\\
&& 0 & 0 &.
\end{tikzcd}
\]
Note that every object in this diagram belongs to $\F(\catX)$,
and hence every exact sequence in this diagram is a conflation in $\F(\catX)$.
Then we have $Y\in\catX$ 
by the exact sequence $0 \to Y \to X \to N \to 0$ and the induction hypothesis.
We also have $A \in \catX$
by the exact sequence $0 \to A \to Y \to M \to 0$ and the induction hypothesis.
Therefore, $\catX$ is a torsion-free class of $\F(\catX)$.

(ii)$\imply$(iii):
It is clear.

(iii)$\imply$(i):
It suffices to show that $\catX$ is closed under kernels in $\catA$
since $\catX$ is clearly extension-closed in $\catA$.
Let $f\colon X \to Y$ be a morphism of $\catA$ such that $X,Y\in \catX$.
We have the following two exact sequences in $\catA$:
\[
0\to \Ker f \to X \to \Ima f \to0,\quad
0\to \Ima f \to Y \to \Cok f \to0.
\]
Then we obtain $\Ker f, \Ima f \in \catX$ 
because $X,Y \in \catX \subseteq \catF$ and $\catF$ is a torsion-free class of $\catA$.
Thus, the first exact sequence is a conflation of $\catF$.
This implies $\Ker f \in \catX$ since $\catX$ is a torsion-free class of $\catF$.
\end{proof}

Dually, we obtain the following result,
which means that \emph{a CE-closed subcategory is a torsion class in a torsion class}.
See \cref{prp:subcat when Serre=tors} for its first application.
\begin{prp}\label{prp:CE=tors in tors}
The following are equivalent for a subcategory $\catX$ of an abelian category $\catA$.
\begin{enur}
\item
$\catX$ is a CE-closed subcategory of $\catA$.
\item
$\catX$ is a torsion class of $\T(\catX)$.
\item
There exists a torsion class $\catT$ of $\catA$
such that $\catX$ is a torsion class of $\catT$.
\qed
\end{enur}
\end{prp}


Next, we describe a similar characterization for IE-closed subcategories.
\begin{prp}\label{prp:char IE}
The following are equivalent for a subcategory $\catX$ of an abelian category $\catA$.
\begin{enur}
\item
$\catX$ is an IE-closed subcategory of $\catA$.
\item
$\catX=\T(\catX) \cap \F(\catX)$ holds.
\item
$\catX$ is a torsion-free class of $\T(\catX)$.
\item
$\catX$ is a torsion class of $\F(\catX)$.
\item
There exists a torsion class $\catT$ of $\catA$
such that $\catX$ is a torsion-free class of $\catT$.
\item
There exists a torsion-free class $\catF$ of $\catA$
such that $\catX$ is a torsion class of $\catF$.
\end{enur}
\end{prp}
\begin{proof}
(i)$\imply$(ii):
It is nothing but \cite[Theorem 2.7]{Eno}.

(ii)$\imply$(iv):
Let $0\to A\to X \to B \to 0$ be a conflation in $\F(\catX)$
such that $X \in \catX$.
We have $B \in \T(\catX)$ since $X \in \catX \subseteq \T(\catX)$
and $\T(\catX)$ is a torsion class of $\catA$.
Thus, we have $B \in \T(\catX)\cap\F(\catX)=\catX$,
and we conclude that $\catX$ is a torsion class of $\F(\catX)$.

(iv)$\imply$(vi):
It is clear.

(vi)$\imply$(i):
It suffices to show that $\catX$ is closed under images in $\catA$
since it is clearly extension-closed in $\catA$.
Let $f\colon X \to Y$ be a morphism of $\catA$
such that $X, Y\in \catX$.
Then we have the following two exact sequences in $\catA$:
\[
0\to \Ker f \to X \to \Ima f \to0,\quad
0\to \Ima f \to Y \to \Cok f \to0.
\]
Then we obtain $\Ker f, \Ima f \in \catF$
because $X,Y \in \catX \subseteq \catF$ and $\catF$ is a torsion-free class of $\catA$.
Thus, the first exact sequence is a conflation of $\catF$.
This implies $\Ima f \in \catX$ since $\catX$ is a torsion class of $\catF$.

(ii)$\imply$(iii)$\imply$(v)$\imply$(i):
We omit the proof since it is the dual of (ii)$\imply$(iv)$\imply$(vi)$\imply$(i). 
\end{proof}

As immediate corollaries of the above results, we have the following.
\begin{cor}\label{prp:IKE = Serre in torf}
The following are equivalent for a subcategory $\catX$ of an abelian category $\catA$.
\begin{enur}
\item
$\catX$ is an IKE-closed subcategory of $\catA$.
\item
$\catX$ is a Serre subcategory of $\F(\catX)$.
\item
There exists a torsion-free class $\catF$ of $\catA$
such that $\catX$ is a Serre subcategory of $\catF$.
\qed
\end{enur}
\end{cor}

\begin{cor}
The following are equivalent for a subcategory $\catX$ of an abelian category $\catA$.
\begin{enur}
\item
$\catX$ is an ICE-closed subcategory of $\catA$.
\item
$\catX$ is a Serre subcategory of $\T(\catX)$.
\item
There exists a torsion class $\catT$ of $\catA$
such that $\catX$ is a Serre subcategory of $\catT$.
\qed
\end{enur}
\end{cor}

We often use the following easy observation.
\begin{prp}\label{prp:torf in Serre}
Let $\catA$ be an abelian category and $\catS$ its Serre subcategory.
Then every torsion(-free) class of $\catS$ is also a torsion(-free) class of $\catA$.
\end{prp}
\begin{proof}
We omit the proof since it is straightforward.
\end{proof}

The results of this section can be summarized as follows:
\begin{align*}
\text{KE-closed} \quad &= \quad \text{torsion-free in torsion-free},\quad
\text{CE-closed} \quad = \quad \text{torsion in torsion},\\
\text{IKE-closed} \quad &= \quad \text{Serre in torsion-free},\quad
\text{ICE-closed} \quad = \quad \text{Serre in torsion},\\
\text{IE-closed} \quad &= \quad \text{torsion in torsion-free} \quad = \quad \text{torsion-free in torsion}.
\end{align*}

\section{Classifying KE-closed subcategories}\label{s:classify KE}
Let $R$ be a commutative noetherian ring.
In this section, we classify KE-closed subcategories of $\catmod R$
when $\dim R \le 1$ (\S \ref{ss:dim 1}) and 
when $R$ is a two-dimensional normal domain (\S \ref{ss:dim 2}).
In the former case, we prove that $\ke(\catmod R)=\torf(\catmod R)$ holds
(see \cref{KE=torf in 1-dim}).
For this, we establish some methods to study KE-closed subcategories of $\catmod R$.

Since KE-closed subcategories and torsion-free classes are closely related 
by \cref{prp:KE=torf in torf},
we first recall Takahashi's classification of the torsion-free classes of $\catmod R$.
Consider the following assignments:
\begin{itemize}
\item
For a subset $\Phi$ of $\Spec R$,
define a subcategory $\catmod^{\ass}_{\Phi} R$ of $\catmod R$ by
\[
\catmod^{\ass}_{\Phi} R:=\{M\in \catmod R \mid \Ass M\subseteq \Phi\}.
\]
\item
For a subcategory $\catX$ of $\catmod R$,
define a subset $\Ass \catX$ of $\Spec R$ by
\[
\Ass\catX :=\bigcup_{M\in\catX} \Ass M.
\]
\end{itemize}
Let $\Phi$ be a subset of $\Spec R$.
A subset $Z$ of $\Phi$ is called an \emph{upper set} of $\Phi$
if for any $\pp,\qq \in \Phi$ with $\pp \subseteq \qq$, 
we have that $\pp\in Z$ implies $\qq \in Z$.
We denote by $\upper \Phi$ the set of upper set of $\Phi$.

\begin{fct}[{\cite[Theorem 4.1]{Takahashi}}]\label{fct:Takahashi}
Let $R$ be a commutative noetherian ring.
\begin{enua}
\item
The assignments $\catX \mapsto \Ass \catX$ and $\Phi \mapsto \catmod^{\ass}_{\Phi} R$ 
give rise to mutually inverse bijections between 
$\torf(\catmod R)$ and the power set $\pow(\Spec R)$ of $\Spec R$.
\item
For a subset $\Phi$ of $\Spec R$,
it is an upper set of $\Spec R$
if and only if $\catmod^{\ass}_{\Phi} R$ is a Serre subcategory of $\catmod R$.
Thus, the bijection in $\mathrm{(1)}$ restricts to a bijection between $\serre(\catmod R)$ and $\upper(\Spec R)$.
\end{enua}
\end{fct}

The following examples and facts will be used frequently throughout this paper.
\begin{ex}\label{ex:cm R is torf}
Let $R$ be a commutative noetherian ring.
\begin{enua}
\item
For a subcategory $\catX$ of $\catmod R$,
we have $\F(\catX)=\catmod^{\ass}_{\Ass \catX} R$ 
(see, for example, \cite[Corollary 3.23]{Sai2}).
Thus, $\catX$ is KE-closed if and only if 
it is a torsion-free class of $\catmod^{\ass}_{\Ass \catX} R$ by \cref{prp:KE=torf in torf}.
\item
If $R$ is Cohen-Macaulay, we have that $\Ass (\cm R) = \Ass R = \Min R$
(see, for example, \cite[Propositions A.12 and A.21]{Sai2}).
Moreover, if $\dim R \le 1$, then $\cm R$ is a torsion-free class of $\catmod R$.
Thus, we have $\cm R = \catmod^{\ass}_{\Min R} R$.
\end{enua}
\end{ex}

\begin{fct}[{\cite[Lemma 2.4 (2)]{IMST}}]\label{fct:KE-closed imply Hom-ideal}
Let $R$ be a commutative noetherian ring 
and $\catX$ a KE-closed subcategory of $\catmod R$.
Then, for any $X\in \catX$ and $M\in \catmod R$,
the $R$-module $\Hom_R(M,X)$ belongs to $\catX$.
\end{fct}

\begin{fct} \label{ass hom}
Let $R$ be a commutative noetherian ring and let $M,N \in \catmod R$.
\begin{enua}
\item
$\Ass_R(\Hom_R(M,N))=\Supp M \cap \Ass N$.
\item
Suppose that $R$ is local.
Then $\depth \Hom_R(M,N) \ge \min\{2, \depth N\}$.
\end{enua}
\end{fct}
\begin{proof}
See \cite[Chapter IV, \S 2.1, Proposition 10]{Bourbaki CA} for (1).
The same proof as in \cite[Lemma 3.1]{BD} works for (2). 
\end{proof}

\begin{fct}[{\cite[Lemma 2.27]{Goto-Watanabe}}]\label{fct:Goto-Watanabe}
Let $R$ be a commutative noetherian ring and $M \in \catmod R$.
Consider a decomposition $\Ass M = \Phi \sqcup \Psi$ as a set.
Then there exists an exact sequence $0\to L \to M\to N\to 0$
such that $\Ass L=\Phi$ and $\Ass N=\Psi$.
\end{fct}

\begin{proof}
Since \cite{Goto-Watanabe} is written in Japanese,
we write down the proof explicitly.
We may assume that both $\Phi$ and $\Psi$ are nonempty.
There is a maximal element $L$ of 
the set of submodules $X$ of $M$ such that $\Ass X \subseteq \Phi$
because $R$ is noetherian.
We want to prove that $\Ass L = \Phi$ and $\Ass (M/L)=\Psi$.
Take $\pp \in \Ass(M/L)$.
There is a submodule $L'$ of $M$ containing $L$
such that $L'/L \iso R/\pp$.
In other words, we have the following exact sequence:
\begin{equation}\label{eq:Goto-Watanabe}
0 \to L \to L' \to R/\pp \to 0.
\end{equation}
If $\pp \in \Phi$, then $\Ass L'\subseteq \Ass L \cup \{\pp\} \subseteq \Phi$.
This contradicts the maximality of $L$,
and we have that $\Ass(M/L) \cap \Phi = \emptyset$.
This means $\Ass(M/L) \subseteq \Psi$.
For the same reason, we have that $\Ass L \subsetneq \Ass L'$.
Thus $\Ass L'=\Ass L \cup \{\pp\}$
since $\Ass L \subsetneq \Ass L' \subseteq \Ass L \cup \{\pp\}$,
In particular, we have $\pp \in \Ass L' \subseteq \Ass M$.
This implies $\Ass(M/L) \subseteq \Ass M$.
Therefore, we have $\Ass M =\Ass L \cup \Ass(M/L)$.
Combining this with $\Ass L \subseteq \Phi$ and $\Ass(M/L)\subseteq \Phi$,
we obtain $\Ass L = \Phi$ and $\Ass(M/L) = \Phi$.
\end{proof}

Let us study when KE-closed subcategories are torsion-free classes.
We will give the prime-idealwise criterion for KE-closed subcategories
to be torsion-free classes in \cref{prp:KE=torf prime-idealwise}.
\begin{lem}\label{prp:mod^ass gen}
Let $R$ be a commutative noetherian ring 
and $\catX$ an extension-closed subcategory of $\catmod R$.
Set $\Phi := \Ass \catX$.
Then the following are equivalent:
\begin{enur}
\item
$\catX$ is a torsion-free class of $\catmod R$.
\item
$\catX = \catmod^{\ass}_{\Phi} R$.
\item
$\catX \supseteq \catmod^{\ass}_{\Phi} R$.
\item
$\catX \supseteq \catmod^{\ass}_{\{\pp\}} R$
for any $\pp \in \Phi$.
\end{enur}
\end{lem}
\begin{proof}
It easily follows from \cref{fct:Goto-Watanabe}.
\end{proof}

\begin{lem}\label{prp:restrict dense torf}
Let $R$ be a commutative noetherian ring
and $\catX$ a KE-closed subcategory of $\catmod R$.
Let $\Phi$ be a subset of $\Ass \catX$.
Then $\catX\cap \catmod^{\ass}_{\Phi} R$ is a KE-closed subcategory of $\catmod R$
such that $\Ass\left(\catX\cap \catmod^{\ass}_{\Phi} R\right)=\Phi$
(or equivalently, $\F\left(\catX\cap \catmod^{\ass}_{\Phi} R\right)=\catmod^{\ass}_{\Phi} R$).
\end{lem}
\begin{proof}
The subcategory $\catX\cap \catmod^{\ass}_{\Phi} R$ is KE-closed in $\catmod R$
since it is the intersection of two KE-closed subcategories.
We prove $\Ass\left(\catX\cap \catmod^{\ass}_{\Phi} R\right)=\Phi$.
It is clear that $\Ass\left(\catX\cap \catmod^{\ass}_{\Phi} R\right) \subseteq \Phi$.
Take $\pp \in \Phi$.
Then there exists $M \in \catX$ such that $\pp \in \Ass M$
since $\Phi \subseteq \Ass \catX$.
By \cref{fct:Goto-Watanabe}, we obtain an exact sequence
$0\to L \to M\to N\to 0$ in $\catmod R$
such that $\Ass L=\{\pp\}$ and $\Ass N=\Ass M \setminus \{\pp\}$.
This exact sequence is a conflation of $\F(\catX)=\catmod^{\ass}_{\Ass \catX} R$
since $\Ass L, \Ass N \subseteq \Ass M \subseteq \Ass \catX$.
Because $\catX$ is a torsion-free class of $\F(\catX)$ by \cref{prp:KE=torf in torf},
we have that $L\in \catX$, and hence $L \in \catX \cap \catmod^{\ass}_{\Phi} R$.
Thus, we have $\pp \in \Ass L \subseteq \Ass\left(\catX\cap \catmod^{\ass}_{\Phi} R\right)$,
and hence $\Phi \subseteq \Ass\left(\catX\cap \catmod^{\ass}_{\Phi} R\right)$.
This finishes the proof.
\end{proof}

Let $R$ be a commutative noetherian ring and $\Phi$ a subset of $\Spec R$.
We define
\[
\ke(R;\Phi):=\{\catX \in \ke(\catmod R) \mid \Ass \catX = \Phi\}.
\]
Note that $\ke(R;\Phi) \supseteq \{\catmod^{\ass}_{\Phi} R\}$ always holds.

\begin{lem}\label{prp:ke ass max ideal}
Let $R$ be a commutative noetherian ring.
Then $\ke(R;\{\mm\})=\{\catmod^{\ass}_{\{\mm\}}R\}$ holds
for any $\mm \in \Max R$.
\end{lem}
\begin{proof}
Let $\catX$ be a KE-closed subcategory of $\catmod R$ 
such that $\Ass\catX \subseteq \Max R$.
Then $\F(\catX)=\catmod^{\ass}_{\Ass\catX} R$ is a Serre subcategory of $\catmod R$
since $\Ass\catX$ is an upper set of $\Spec R$.
From this, $\catX$ is a torsion-free class of $\catmod R$ 
by \cref{prp:torf in Serre,prp:KE=torf in torf},
and hence $\catX=\catmod^{\ass}_{\Ass\catX}R$.
In particular,
we have $\ke(R; \{\mm\})=\{\catmod^{\ass}_{\{\mm\}}R\}$
for any $\mm \in \Max R$.
\end{proof}

\begin{prp}\label{prp:KE=torf prime-idealwise}
Let $R$ be a commutative noetherian ring
and $\Phi$ a subset of $\Spec R$.
If $\ke(R;\{\pp\}) = \{\catmod^{\ass}_{\{\pp\}} R\}$ holds for any $\pp \in \Phi$,
then $\ke(R;\Phi) = \{\catmod^{\ass}_{\Phi} R\}$ also holds.
\end{prp}
\begin{proof}
Let $\catX$ be a KE-closed subcategory of $\catmod R$ such that $\Ass \catF=\Phi$.
Then we have $\catX\supseteq \catX \cap \catmod^{\ass}_{\{\pp\}}R =\catmod^{\ass}_{\{\pp\}} R$
for any $\pp \in \Phi$ by the assumption and \cref{prp:restrict dense torf}.
Thus $\catX$ is also a torsion-free class of $\catmod R$ by \cref{prp:mod^ass gen}.
\end{proof}

\begin{cor}\label{prp:char KE = torf}
The following are equivalent for a commutative noetherian ring $R$.
\begin{enur}
\item
$\ke(\catmod R)=\torf(\catmod R)$ holds.
\item
$\ke(R;\{\pp\})= \{ \catmod^{\ass}_{\{\pp\}}R \}$ holds for any $\pp \in \Spec R$.
\item
$\ke(R;\{\pp\})= \{ \catmod^{\ass}_{\{\pp\}}R \}$ holds for any $\pp \in \Spec R\setminus \Max R$.
\end{enur}
\end{cor}
\begin{proof}
(i)$\imply$(ii)$\imply$(iii):
It is clear.

(ii)$\imply$(i):
It follows from \cref{prp:KE=torf prime-idealwise}.

(iii)$\imply$(ii):
It follows from \cref{prp:ke ass max ideal}.
\end{proof}

As an immediate corollary, we obtain a classification of KE-closed subcategories of $\catmod R$
when $R$ is a commutative artinian ring.
\begin{cor}\label{prp:KE=torf when dim 0}
If $R$ is a commutative artinian ring,
then we have $\ke(\catmod R)=\torf(\catmod R)$.
\end{cor}
\begin{proof}
It follows from \cref{prp:char KE = torf} (iii) since $\Spec R = \Max R$.
\end{proof}
We will use \cref{prp:char KE = torf} to prove KE-closed subcategories and torsion-free classes coincide in $\catmod R$ when $\dim R\le 1$
(cf.\ \cref{KE=torf in 1-dim,prp:KE=torf when 1-dim CM with can mod}).

Next, we study the relationship between KE-closed subcategories of $\catmod R$ and $\catmod R/I$
for an ideal $I$ of $R$. 
Let $F\colon \catA \to \catB$ be a functor between abelian categories.
For a subcategory $\catX$ of $\catB$,
we denote by $F^{-1}(\catX)$ the subcategory of $\catA$ 
consisting of $A\in \catA$ such that $FA\in \catX$.
\begin{lem}
Let $F\colon \catA \to \catB$ be an additive functor between abelian categories
and $\catX$ an additive subcategory of $\catB$.
\begin{enua}
\item
If $F$ is left exact and $\catX$ is closed under subobjects (resp.\ kernels),
then so is $F^{-1}(\catX)$.
\item
If $F$ is right exact and $\catX$ is closed under quotients (resp.\ cokernels),
then so is $F^{-1}(\catX)$.
\item
If $F$ is exact and $\catX$ is closed under extensions,
then so is $F^{-1}(\catX)$.
\end{enua}
\end{lem}
\begin{proof}
We omit the proof since it is straightforward.
\end{proof}

Let $\phi\colon R\to S$ be a finite homomorphism of commutative noetherian rings.
Then the restriction functor $\pi \colon \catmod S \to \catmod R$ is exact.
Thus, if $\catX$ is a KE-closed subcategory (resp.\ torsion-free class) of $\catmod R$,
then so is $\pi^{-1}(\catX)$ in $\catmod S$.
It is natural to ask what $\Ass (\pi^{-1}(\catX))$ is.

\begin{prp}\label{prp:inverse image torf}
Let $\phi\colon R\to S$ be a finite homomorphism of commutative noetherian rings
and $f\colon \Spec S \to \Spec R$ the natural map induced by $\phi$. 
Consider the restriction functor $\pi \colon \catmod S \to \catmod R$.
Then we have that $\pi^{-1}(\catmod^{\ass}_{\Phi} R) = \catmod^{\ass}_{f^{-1}(\Phi)} S$
for any subset $\Phi$ of $\Spec R$.
\end{prp}
\begin{proof}
Recall that $\Ass_R(M_R)= f(\Ass_S(M))$ for any $M\in \catmod S$.
The proposition follows from the following:
\begin{align*}
M \in \pi^{-1}(\catmod^{\ass}_{\Phi} R)
& \Equi M_R \in \catmod^{\ass}_{\Phi} R 
\Equi f(\Ass_S(M)) = \Ass_R(M_R) \subseteq \Phi\\
&\Equi \Ass_S(M) \subseteq f^{-1}(\Phi)
\Equi M \in \catmod^{\ass}_{f^{-1}(\Phi)} S.
\end{align*}
\end{proof}

Let $I$ be an ideal of a commutative ring $R$.
We denote by $V(I)$ the set of prime ideals of $R$ containing $I$.
For any $\pp\in V(I)$,
we denote by $\ol{\pp}$ the prime ideal of $R/I$ corresponding to $\pp$.

\begin{prp}\label{prp:Ass catX  R/I}
Let $I$ be an ideal of a commutative noetherian ring $R$
and $\pi \colon \catmod(R/I) \to \catmod R$ the restriction functor.
Let $\catX$ be a KE-closed subcategory of $\catmod R$.
Then the natural map $f\colon \Spec(R/I) \inj \Spec R$ induces
a bijection $\Ass(\pi^{-1}(\catX)) \isoto \Ass \catX \cap V(I)$.
\end{prp}
\begin{proof}
The injective map $f$ restricts to the map $\Ass(\pi^{-1}(\catX)) \inj \Ass \catX \cap V(I)$
since $f$ induces the bijection $\Ass_{R/I} M \isoto \Ass_R M$ for any $M \in \catmod R/I$.
It is enough to show that the map $\Ass(\pi^{-1}(\catX)) \inj \Ass \catX \cap V(I)$ is surjective.
Take $\pp \in \Ass \catX \cap V(I)$.
Then there exists $M\in\catX$ such that $\pp \in \Ass M$.
Then $\Hom_R(R/I, M) \in \catX$ by \cref{fct:KE-closed imply Hom-ideal}
and $\Ass_R \Hom_R(R/I, M) = \Ass M \cap V(I) \ni \pp$ by \cref{ass hom} (1).
Thus, we have $\ol{\pp} \in \Ass_{R/I} \Hom_R(R/I, M) \subseteq \Ass(\pi^{-1}(\catX))$
and obtain the desired conclusion.
\end{proof}

The following proposition is useful in reducing 
the study of KE-closed subcategories to a simpler case.
\begin{prp} \label{prp:surj base change}
Let $I$ be an ideal of a commutative noetherian ring $R$
and $\pi \colon \catmod(R/I) \to \catmod R$ the restriction functor.
Let $\catX$ be an extension-closed subcategory of $\catmod R$.
\begin{enua}
\item
For any $\pp\in V(I)$,
if $\pi^{-1}(\catX) \supseteq \catmod^{\ass}_{\{\ol{\pp}\}}(R/I)$,
then $\catX \supseteq \catmod^{\ass}_{\{\pp\}}R$.
\item
Let $\catF$ be a torsion-free class of $\catmod R$ such that $\Ass \catF \subseteq V(I)$.
If $\pi^{-1}(\catX) \supseteq \pi^{-1}(\catF)$,
then $\catX \supseteq \catF$.
\end{enua}
\end{prp}
\begin{proof}
(1)
The proof is inspired by \cite[Lemma 4.2]{Takahashi}.
To obtain a contradiction, suppose that $\catX \not\supseteq \catmod^{\ass}_{\{\pp\}}R$.
There exists $M\in \catmod R$ such that $\Ass M =\{\pp\}$ and $M\not\in \catX$.
Let $f_1,\dots,f_n$ be a system of generators of $\Hom_R(M,R/\pp)$ as an $R$-module.
Put $f:=\begin{bsmatrix}
f_1 & \cdots & f_n
\end{bsmatrix}^t \colon M \to (R/\pp)^{\oplus n}$
and $M_1:=\Ker(f)$.
Then $\Ima(f)_{(R/I)} \in \catmod^{\ass}_{\{\ol{\pp}\}}(R/I)$
since $(R/\pp)_{(R/I)} \iso (R/I)/(\pp/I) \in \catmod^{\ass}_{\{\ol{\pp}\}}(R/I)$.
From this, we have that $\Ima(f) \in \catX$
since $\catmod^{\ass}_{\{\ol{\pp}\}}(R/I) \subseteq \pi^{-1}(\catX)$.
Because $\catX$ is extension-closed and 
there is an exact sequence $0\to M_1 \to M \to \Ima(f)\to 0$,
we have that $M_1 \not\in \catX$.
In particular, $M_1$ is nonzero, and hence $\Ass M_1= \{\pp\}$.

Iterating this procedure, we obtain a descending chain 
$M=M_0 \supseteq M_1 \supseteq M_2 \supseteq \cdots $ of $R$-submodules of $M$
such that $\Ass M_i=\{\pp\}$ and $M_i \not\in \catX$ for any $i$.
The $R_{\pp}$-module $M_{\pp}$ is of finite length since $\Ass M=\{\pp\}$.
Thus, the descending chain 
$M_{\pp}=M_{0,\pp} \supseteq M_{1,\pp} \supseteq M_{2,\pp} \supseteq \cdots$
becomes stationary, that is, there is some $n \ge 0$ such that 
$M_{n,\pp}=M_{n+1,\pp}=M_{n+2,\pp}=\cdots$.
By the construction of $M_n$,
we can conclude that $\Hom_{R_{\pp}}(M_{n,\pp},\kk(\pp))=\Hom_R(M_n,R/\pp)_{\pp}=0$,
and hence $M_{n,\pp}=0$.
This contradicts the fact that $\pp \in \Ass M_n \subseteq \Supp M_n$.

(2)
Put $\Phi :=\Ass \catF \subseteq \Spec R$ 
and $\ol{\Phi}:=\{\ol{\pp} \mid \pp\in\Phi\} \subseteq \Spec(R/I)$.
Then $\catF=\catmod^{\ass}_{\Phi} R$ and $\pi^{-1}(\catF)=\catmod^{\ass}_{\ol{\Phi}}(R/I)$
by \cref{prp:inverse image torf} and the assumption that $\Phi = \Ass \catF \subseteq V(I)$.
Thus (2) follows from (1) and \cref{prp:mod^ass gen}.
\end{proof}

\subsection{The one-dimensional case}\label{ss:dim 1}
In this subsection, we classify the KE-closed subcategories of $\catmod R$ 
when $R$ is a commutative noetherian ring with $\dim R \le 1$.
We will see that KE-closed subcategories coincide with torsion-free classes in this case
(\cref{KE=torf in 1-dim}).
Moreover, the converse holds, that is,
if KE-closed subcategories coincide with torsion-free classes,
then $\dim R \le 1$ for a homomorphic image $R$ of a Cohen-Macaulay ring
(\cref{prp:ke=torf imply dim=1}).
Let us start with the latter direction.
\begin{lem}\label{depth at most 1}
Let $R$ be a commutative noetherian ring.
If $\ke(\catmod R)=\torf(\catmod R)$ holds,
then we have $\depth_{R_\pp} M_{\pp} \le 1$ for any $R$-module $M\in \catmod R$ and $\pp \in \Supp M$.
\end{lem}

\begin{proof}
Take $M\in \catmod R$ and $\pp \in \Supp M$.
Consider the subcategory $\catX$ consisting of $N \in \catmod R$ 
such that $\depth_{R_\pp} N_{\pp} \ge 2$.
It is a KE-closed subcategory of $\catmod R$ by the depth lemma.
Then $\catX$ is a torsion-free class of $\catmod R$ by the assumption.
Suppose that $\depth_{R_\pp} M_{\pp} \ge 2$, that is, $M \in \catX$.
Then we have $\pp M \in \catX$ since $\catX$ is closed under subobjects.
On the other hand, Nakayama's lemma implies that $(M/\pp M)_{\pp}$ is nonzero.
Therefore, applying the depth lemma to the short exact sequence $0 \to \pp M \to M \to M/\pp M \to 0$
yields that $\depth_{R_\pp}(\pp M)_{\pp}=1$.
This is a contradiction.
Therefore, we have $\depth_{R_\pp} M_{\pp} \le 1$.
\end{proof}

The dimension of a commutative noetherian ring such that $\ke(\catmod R)=\torf(\catmod R)$
is constrained by this lemma as follows.
\begin{cor}\label{prp:ke=torf imply dim=1}
Let $R$ be a commutative noetherian ring with $\ke(\catmod R)=\torf(\catmod R)$.
\begin{enua}
\item
In general, we have $\dim R \le 2$.
\item
If $R$ is a homomorphic image of a Cohen-Macaulay ring,
then we have $\dim R \le 1$.
\end{enua}
\end{cor}
\begin{proof}
(1)
For any maximal ideal $\mm$ of $R$,
we have the following inequalities:
\[
\dim R_{\mm}-1 
\le \sup\{\depth R_{\pp} \mid \pp \in \Spec R,\; \pp \subseteq \mm\}
\le \sup\{\depth R_{\pp} \mid \pp \in \Spec R\} \le 1,
\]
where the first and third inequalities follow from \cite[Lemma 1.4]{CFF02}
and \cref{depth at most 1}, respectively.
Thus we obtain $\dim R \le 2$.

(2)
We have a Cohen-Macaulay ring $S$ and a surjective map $\pi \colon S \to R$.
Fix a prime ideal $\pp\in \Spec R$.
Let $\qq$ be the prime ideal $\pi^{-1}(\pp)$ of $S$.
We may take an ideal $I\subseteq \Ker \pi$ such that $I_\qq$ is generated by a maximal regular sequence in $(\Ker \pi)_\qq$.
Then
\[
\depth_{S_\qq} (S/I)_\qq=\dim (S/I)_\qq =\dim S_\qq -\grade (I_\qq, S_\qq) =\dim S_\qq -\grade \left((\Ker \pi)_\qq, S_\qq\right)=\dim R_\pp.
\]
Here the second and fourth equalities follow from \cite[Theorem 2.1.2(b)]{BH} since $S$ is Cohen-Macaulay.
In particular, $\Hom_S(R,S/I)_\qq$ is nonzero by \cref{ass hom}(1).
Note that 
\[
\min\{2, \dim R_\pp\} \le \depth_{S_\qq} \Hom_S(R,S/I)_\qq=\depth_{R_\pp} \Hom_S(R,S/I)_\pp \le 1.  
\]
Here the first and the third inequalities follow from \cref{ass hom}(2) and \cref{depth at most 1} respectively.
This shows that $\dim R_\pp \le 1$.
Since $\pp$ can be any prime ideal of $R$, we conclude that $\dim R\le 1$.
\end{proof}

\begin{rmk}
There exists a $2$-dimensional commutative noetherian local domain $R$ having no nonzero maximal Cohen-Macaulay $R$-modules; see \cite[Example 2.1]{LW}.
For such a ring, we can't apply the same kind of proof of \cref{prp:ke=torf imply dim=1}.
Thus, without assuming that $R$ is a homomorphic image of a Cohen-Macaulay ring, we don't know whether the equality $\ke(\catmod R)=\torf(\catmod R)$ forces $R$ to have $\dim R \le 1$ or not.
\end{rmk}

We introduce the definition of dominant resolving subcategories and record its characterization.
This class of subcategories plays a key role in the rest of this section.

\begin{dfn} \label{dfn:dom res}
Let $R$ be a commutative noetherian ring and let $M\in \catmod R$.
Let $\catX$ be a subcategory of $\catmod R$.
\begin{enua}
\item For an integer $n\ge 0$, We denote by $\Omega^n_RM$ the $n$-th syzygy module of $M$ in some $R$-projective resolution of $M$.
Note that $\Omega^n_RM$ is uniquely determined up to projective summands.
\item $\catX$ is said to be \emph{resolving} if it contains all projective $R$-modules, is additive, and is closed under direct summands, extensions, and kernels of epimorphisms.
\item $\catX$ is
said to be \emph{dominant} if for all $\pp\in\Spec R$, there exists an integer $n\ge 0$ such that $\Omega^n_R(R/\pp)\in \catX$.
\item
We denote by $\domres(\catmod R)$ the set of dominant resolving subcategories of $\catmod R$.
\end{enua}
\end{dfn}

\begin{thm}[{\cite[Corollary 4.6 and Proposition 5.3(1)]{Takahashi2}}] \label{thm:dom res}
Let $R$ be a commutative noetherian ring and let $\catX$ be a resolving subcategory of $\catmod R$.
Then the following are equivalent.
\begin{enua}
\item $\catX$ is dominant.
\item The equality below holds true:
\[
\catX=\{M\in\catmod R\mid \depth M_\pp \ge \inf_{X\in\catX}\{\depth_{R_\pp} X_\pp\}\text{ for all }\pp\in\Spec R\}.
\]
\end{enua}
\end{thm}

Next, we investigate the relationship between dominant resolving subcategories and KE-closed subcategories.

\begin{lem} \label{lem:2 syzygy}
Let $R$ be a commutative noetherian ring and 
let $\catX$ be an additive subcategory of $\catmod R$.
Assume $\catX$ contains $R$ and is closed under kernels.
Then for any $M\in\catmod R$, $\Omega^2_RM\in \catX$.
\end{lem}

\begin{proof}
Let $M\in\catmod R$.
There exist projective $R$-modules $P, Q$ and an exact sequence
\[
0 \to \Omega^2_RM \to P \to Q \to M \to 0.
\]
Since $\catX$ is an additive subcategory closed under direct summands by \cref{rmk:closed under DS} and possesses $R$, $\catX$ contains all projective $R$-module.
In particular, both $P$ and $Q$ belong to $\catX$.
Then, as $\catX$ is closed under kernels, $\Omega^2_RM\in\catX$.
\end{proof}

\begin{cor} \label{KE-closed res}
Let $R$ be a commutative noetherian ring and let $\catX$ be a subcategory of $\catmod R$.
Then the following are equivalent.
\begin{enua}
\item $\catX$ is a KE-closed subcategory, and $R\in\catX$.
\item $\catX$ is a dominant resolving subcategory, and
\[
\inf_{X\in\catX}\{\depth_{R_\pp} X_\pp\} \le 2 \text{ for all }\pp\in\Spec R.
\]
\end{enua}
\end{cor}

\begin{proof}
(1)$\Rightarrow$(2): By definition, it follows that $\catX$ is resolving.
Due to \cref{lem:2 syzygy}, $\catX$ contains $\Omega^2_R(R/\pp)$ for all $\pp\in\Spec R$.
This shows that $\catX$ is dominant.
Suppose that there exists a prime ideal $\pp$ such that $\inf_{X\in\catX}\{\depth_{R_\pp} X_\pp\} \ge 2$.
In particular, $\depth R_\pp\ge 2$.
Take a part of an $R$-free resolution of $R/\pp$ so that the following sequence is exact:
\[
0 \to \Omega^2_R(R/\pp) \to F_1 \to F_0 \to R/\pp \to 0.
\]
Here $F_1$ and $F_0$ are $R$-free.
Since $R\in\catX$ and $\catX$ is KE-closed, $\Omega^2_R(R/\pp)\in\catX$.
Applying the depth lemma to the localization at $\pp$ of the sequence above yields that $\depth_{R_\pp}(\Omega^2_R(R/\pp))_\pp=2$.
Therefore, one gets the inequality $\inf_{X\in\catX}\{\depth_{R_\pp} X_\pp\} \le \depth_{R_\pp}(\Omega^2_R(R/\pp))_\pp=2$.

(2)$\Rightarrow$(1): This follows from \cref{thm:dom res} and the depth lemma.
\end{proof}

\begin{cor}
Let $R$ be a commutative noetherian ring and let $\catX$ be a subcategory of $\catmod R$.
Suppose $\dim R\le 2$.
Then the following are equivalent.
\begin{enua}
\item $\catX$ is a KE-closed subcategory, and $R\in\catX$.
\item $\catX$ is a dominant resolving subcategory.
\end{enua}
\end{cor}

\begin{proof}
By the assumption, we have 
\[
\inf_{X\in\catX}\{\depth_{R_\pp} X_\pp\} \le \depth R_\pp \le \dim R_\pp \le 2 \text{ for all }\pp\in\Spec R.
\]
Thus the assertion follows from \cref{KE-closed res}.
\end{proof}

\begin{cor} \label{torf res}
Let $R$ be a commutative noetherian ring and let $\catX$ be a subcategory of $\catmod R$.
Then the following are equivalent.
\begin{enua}
\item $\catX$ is a torsion-free class, and $R\in\catX$.
\item $\catX$ is a dominant resolving subcategory, and
\[
\inf_{X\in\catX}\{\depth_{R_\pp} X_\pp\} \le 1 \text{ for all }\pp\in\Spec R.
\]
\end{enua}
\end{cor}

\begin{proof}
By considering 1st syzygies of $R/\pp$ instead of 2nd syzygies, we may prove the assertion using an argument similar to \cref{KE-closed res}.
\end{proof}

\begin{fct}[{\cite[Example 4.5]{IMST}}] \label{imst}
Let $R$ be a commutative noetherian ring and let $\catX$ be a subcategory of $\catmod R$ which is additive and closed under extensions.
Suppose $\mm\in \catX$ for some maximal ideal $\mm$ of $R$.
Then $R\in\catX$.
\end{fct}

\begin{cor} \label{dom res over dim 1}
Let $R$ be a commutative noetherian ring and let $\catX$ be a subcategory of $\catmod R$.
Suppose $\dim R\le 1$.
Then the following are equivalent.
\begin{enua}
\item $\catX$ is a torsion-free class, and $R\in\catX$.
\item $\catX$ is a torsion-free class and contains all torsion-free $R$-modules.
\item $\catX$ is a KE-closed subcategory, and $R\in\catX$.
\item $\catX$ is a KE-closed subcategory, and $\mm\in\catX$ for some maximal ideal $\mm$ of $R$.
\item $\catX$ is a dominant resolving subcategory of $\catmod R$.
\end{enua}
\end{cor}

\begin{proof}
The implications (2)$\Rightarrow$(1)$\Rightarrow$(3) and (2)$\Rightarrow$(4) are obvious.

The implications (4)$\Rightarrow$(3) and (3)$\Rightarrow$(5) follow from \cref{imst} and \cref{KE-closed res} respectively.

Thus, what remains to show is the implication (5)$\Rightarrow$(2).
Note that for a prime ideal $\pp\in\Spec R$, $\pp \in \Ass \catX$ exactly when there exists $M\in\catX$ such that $\depth_{R_{\pp}}M_\pp=0$.
In other words, $\pp\in \Ass \catX$ if and only if $\inf_{X\in\catX} \{\depth_{R_\pp}X_\pp\}=0$.
Therefore, \cref{thm:dom res} shows that $\catX$ contains $\catmod_{\{\pp\}}^\ass R$ for all $\pp\in \Ass \catX$.
It follows from \cref{prp:mod^ass gen} that $\catX$ is a torsion-free class.
Also, we remark that, since $\dim R\le 1$, $\catmod_{\Ass R}^\ass R$ consists of all torsion-free $R$-modules.
From this and the fact that $R\in\catX$, we see that $\catX$ contains all torsion-free $R$-modules.
\end{proof}

In the following, we prove the lemmas needed to prove the remaining part of \cref{thma:KE=torf}.
Let $R$ be a commutative noetherian ring.
$Q(R)$ denotes the total ring of quotients of $R$.
Recall that a finitely generated $R$-submodule $I$ of $Q(R)$ is called a \emph{fractional ideal} if $I$ contains a regular element of $R$.
Remark that fractional ideals are isomorphic
to some ideals of $R$, and hence are torsion-free. Also remark that if $I$ and $J$ are fractional
ideals, then $\Hom_R(I, J)$ is isomorphic to the fractional ideal $J : I = \{a\in Q(R)\mid aI\subseteq J\}$ (\cite[Lemma 2.4.3]{HS}).

\begin{prp} \label{seq of subalg}
Let $R$ be a commutative noetherian ring with $\dim R \le 1$.
Let $S\in\catmod(R)$ be an $R$-subalgebra of $Q(R)$.
Then there exists a sequence of $R$-subalgebras $R=S_0 \subsetneq S_1 \subsetneq \cdots \subsetneq S_n=S$ such that for each $i=0,\dots,n-1$, there exists $\mm_i\in\Max(S_i)$ such that $S_{i+1}\subseteq \mm_i:\mm_i$.
\end{prp}

\begin{proof}
Note that $S/R$ is a torsion $R$-module.
Therefore, since $\dim R\le 1$, $S/R$ is an $R$-module of finite length.
Now, we prove by induction on the length $l$ of $S/R$.
If $l=0$, then there is nothing to prove.
Suppose $l>0$.
Take a maximal ideal $\mm$ of $R$ which belongs to $\Ass(S/R)$.
By the definition of associated primes, there exists an element $x\in S\setminus R$ such that $\mm x\subseteq R$.
We then have $\mm x \subseteq \mm$; otherwise, the multiplication by $x$ gives a surjection $\mm R_\mm \to R_\mm$. 
This implies that $R_\mm$ is a discrete valuation ring.
As $S$ is integral over $R$, $S_\mm$ must be equal to $R_\mm$ (Here, there is a canonical inclusion from $Q(R)_\mm$ to $Q(R_\mm)$, so that $S_\mm$ can be regarded as a subring of $Q(R_\mm)$).
However, since $\mm\in\Ass(S/R)$, $(S/R)_\mm$ is nonzero, which shows a contradiction.

We see that $x\in \mm:\mm$.
Since $\mm:\mm$ is an $R$-algebra, $S_1:=R[x]$ is a subalgebra of $\mm:\mm$.
Obviously, $S_1$ is not equal to $R$ and contained in $S$.
It follows that the length of the $S_1$-module $S/S_1$ is less than $l$.
By the induction hypothesis, we get a sequence $S_1\subsetneq S_2 \subsetneq\cdots\subsetneq S_n=S$ of $S_1$-algebras with suitable maximal ideals.
Thus, we achieve a desired sequence $R=S_0 \subsetneq S_1 \subsetneq \cdots \subsetneq S_n=S$ of $R$-algebras with suitable maximal ideals.
\end{proof}

Let $R$ be a commutative noetherian ring.
For an $R$-module $M$, we denote by $\ZZ_R(M)$ the center of the endomorphism ring $\End_R(M)$ of $M$.

\begin{lem} \label{center}
Let $R$ be a commutative noetherian ring and let $\catX$ be a subcategory of $\catmod R$ which is closed under finite direct sums and kernels.
Let $M\in \catX$.
Then $\ZZ_R(M)$ belongs to $\catX$.
\end{lem}

\begin{proof}
Set $A=\End_R(M)$.
$M$ can be regarded as a finitely generated left $A$-module.
Therefore, we may take a part of an $A$-free resolution $F_1 \to F_0 \to M \to 0$ of $M$.
Applying the functor $\Hom_A(-,M)$, we get an exact sequence of $R$-modules
\[
0 \to \Hom_A(M,M) \to \Hom_A(F_0,M) \to \Hom_A(F_1,M).
\]
Since $\catX$ is closed under finite direct sums, both $\Hom_A(F_0,M)$ and $\Hom_A(F_1,M)$ belong to $\mathcal{X}$.
Then, since $\mathcal{X}$ is closed under kernels, $\Hom_A(M,M)$ is also in $\catX$.
On the other hand, $\Hom_A(M,M)$ is equal to $\ZZ_R(M)$ by the assumption that $R$ is commutative.
\end{proof}

\begin{prp} \label{subalgebra}
Let $R$ be a commutative noetherian integral domain and let $M$ be a nonzero torsion-free $R$-module.
Then $\ZZ_R(M)$ is isomorphic to an $R$-subalgebra of $Q(R)$.
\end{prp}

\begin{proof}
$M$ being nonzero torsion-free implies that $M$ is faithful over $R$.
This means that the canonical map $R\to \ZZ_R(M)$ is injective.
It is clear that $\ZZ_R(M)$ is torsion-free as an $R$-module.
Since $\ZZ_R(M)=\Hom_{\End_R(M)}(M,M)$, $\ZZ_R(M)$ commutes with the localization; $\ZZ_R(M)\otimes_R Q(R) \cong\ZZ_{Q(R)}(M\otimes_RQ(R))$.
As $Q(R)$ is a field, $M\otimes_R Q(R)$ is $Q(R)$-free.
Therefore, $\ZZ_R(M)\otimes_R Q(R) \cong Q(R)$ as $Q(R)$-algebras.
By the construction, the composition of the inclusions $R \to \ZZ_R(M)$ and $\ZZ_R(M) \to \ZZ_R(M)\otimes_R Q(R)\cong Q(R)$ is identical to the canonical inclusion $R\to Q(R)$.
This means that $\ZZ_R(M)$ is isomorphic to an $R$-subalgebra of $Q(R)$.
\end{proof}

Now we can prove the remaining part of \cref{thma:KE=torf}.

\begin{thm} \label{KE=torf in 1-dim}
If $R$ is a commutative noetherian ring with $\dim R\le 1$,
then we have $\ke(\catmod R)=\torf(\catmod R)$.
\end{thm}

\begin{proof}
By \cref{prp:char KE = torf}, we only need to verify that
$\ke(R;\{\pp\})= \{ \catmod^{\ass}_{\{\pp\}}R \}$ holds for any $\pp \in \Min R$.
Fix a prime ideal $\pp \in \Min R$.
Let $\catX$ be a KE-closed subcategory of $\catmod R$ such that $\Ass \catX=\{\pp\}$.
In view of \cref{prp:mod^ass gen}, it is enough to show that $\catX \supseteq \catmod^{\ass}_{\{\pp\}} R$.

Let $\pi\colon \catmod(R/\pp) \to \catmod R$ be the restriction functor.
Then by \cref{prp:Ass catX  R/I}, $\pi^{-1}(\catX)$ is a nonzero KE-closed subcategory of $\catmod(R/\pp)$.
Take a nonzero module $M\in \pi^{-1}(\catX)$.
Then, since $\Ass_R(M)=\{\pp\}$, $M$ is torsion-free as an $R/\pp$-module.
\cref{center} says that $\ZZ_{R/\pp}(M)\in \pi^{-1}(\catX)$.
By \cref{subalgebra}, we see that $\ZZ_{R/\pp}(M)$ is isomorphic to an $R/\pp$-subalgebra $S$ of $Q(R/\pp)$.
Using \cref{seq of subalg}, we obtain an integer $n\ge 0$ and a sequence of $R/\pp$-subalgebras $R/\pp=S_0\subsetneq S_1 \subsetneq \cdots \subsetneq S_n=S$ with maximal ideals $\mm_i\in \Max(S_i)$ such that $S_{i+1}\subseteq \mm_i:\mm_i$.

For each $i=1,\dots,n$, consider the restriction functor $\rho_i\colon \catmod S_i \to \catmod R/\pp$.
We claim that $\rho_i^{-1}(\catX)$ contains all torsion-free $S_i$-modules.
We prove this by decreasing induction on $i$.

If $i=n$, then $S_n\cong \ZZ_R(M)\in \catX$.
Since $\rho_n^{-1}(\catX)$ is a torsion-free class in $\catmod S_n$, this, with \cref{dom res over dim 1}, shows that $\rho_n^{-1}(\catX)$ contains all torsion-free $S_n$-modules.

Next suppose that $i<n$.
Since $S_{i+1}\subseteq\mm_i:\mm_i$, $\mm_i$ may be viewed as a torsion-free $S_{i+1}$-module.
By the induction hypothesis, $\mm_i$ belongs to $\rho_{i+1}^{-1}(\catX)$.
This then implies that the $S_i$-module $\mm_i$ belongs to $\rho_{i}^{-1}(\catX)$.
Due to the fact that $\rho_{i}^{-1}(\catX)$ is a KE-closed subcategory of $\catmod S_i$, \cref{dom res over dim 1} yields that $\rho_{i}^{-1}(\catX)$ contains all torsion-free $S_i$-modules.

As a consequence, we obtain that $\rho_0^{-1}(\catX)=\pi^{-1}(\catX)$ contains all torsion-free $R/\pp$-modules.
In other words, $\pi^{-1}(\catX) \supseteq \pi^{-1}(\catmod^{\ass}_{\{\pp\}} R)$.
Thus, we have $\catX \supseteq \catmod^{\ass}_{\{\pp\}} R$ by \cref{prp:surj base change},
and this finishes the proof.
\end{proof}
\subsection{The two-dimensional case}\label{ss:dim 2}
In this subsection, we classify the KE-closed subcategories of $\catmod R$
for a two-dimensional normal domain $R$.
Let us begin with the consequences of the results of the previous subsection 
for the higher dimensional case.

\begin{prp}\label{prp:KE=torf when ht ge dim - 1 locally}
Let $R$ be a commutative noetherian ring of finite Krull dimension.
Then for any $\pp \in \Spec R$ such that $\htt \pp \ge \dim R -1$,
we have $\ke(R; \{\pp\})= \{\catmod^{\ass}_{\{\pp\}} R\}$.
\end{prp}
\begin{proof}
Let $\catX$ be a KE-closed subcategory with $\Ass \catX=\{\pp\}$,
and let $\pi \colon \catmod R/\pp \to \catmod R$ be the restriction functor.
Then $\pi^{-1}(\catX)$ is a KE-closed subcategory of $\catmod R/\pp$
such that $\Ass \left(\pi^{-1}(\catX)\right) = \{(0)\}$
by \cref{prp:Ass catX  R/I}.
Since $\dim R/\pp \le \dim R - \htt \pp \le 1$,
we have that $\pi^{-1}(\catX) = \catmod^{\ass}_{\{(0)\}} R/\pp$
by \cref{KE=torf in 1-dim}.
This implies $\catX = \catmod^{\ass}_{\{\pp\}} R$
by \cref{prp:surj base change} (1).
\end{proof}

\begin{cor}\label{prp:KE=torf when ht ge dim - 1}
Let $R$ be a commutative noetherian ring of finite Krull dimension
and $\catX$ a KE-closed subcategory.
If $\Ass \catX$ consists of prime ideals $\pp$ such that $\htt \pp \ge \dim R -1$,
then $\catX$ is a torsion-free class.
\end{cor}
\begin{proof}
It follows from
\cref{prp:KE=torf prime-idealwise,prp:KE=torf when ht ge dim - 1 locally}.
\end{proof}

\begin{thm} \label{KE over normal with dim 2}
Let $R$ be a noetherian normal domain and let $\catX$ be a KE-closed subcategory of $\catmod R$.
Then
\begin{enua}
\item If $(0)\in\Ass\catX$, then $\catX$ is a dominant resolving subcategory.
\item If $(0)\not\in \Ass \catX$ and $\dim R\le 2$, then $\catX$ is a torsion-free class.
\end{enua}
\end{thm}

\begin{proof}
(1): Assume $(0)\in\Ass\catX$.
We see from \cref{prp:restrict dense torf} that there exists an $R$-module $M\in \catX$ such that $\Ass M=\{(0)\}$.
Note that $M$ is a torsion-free $R$-module.
By \cref{center} and \cref{subalgebra}, $\ZZ_R(M)$ is an $R$-subalgebra of $Q(R)$ such that $\ZZ_R(M)\in \catX$.
By the normality of $R$, $R=\ZZ_R(M)$.
This shows that $R\in\catX$.
Then \cref{KE-closed res} implies that $\catX$ is a dominant resolving subcategory.

(2): It follows from \cref{prp:KE=torf when ht ge dim - 1}.
\end{proof}

\begin{cor}\label{prp:classify KE when two-dim normal local domain}
Let $R$ be a two-dimensional noetherian normal domain.
\begin{enua}
\item
We have $\ke(\catmod R)=\torf(\catmod R) \cup \domres(\catmod R)$.
\item
When $R$ is local,
we have $\ke(\catmod R)=\torf(\catmod R) \cup\{\cm R\}$.
\end{enua}
\end{cor}
\begin{proof}
(1) follows from \cref{KE over normal with dim 2}.
We now prove (2).
Let $\catX$ be a KE-closed subcategory of $\catmod R$ which is not a torsion-free class.
Thanks to Corollary \ref{torf res} and Theorem \ref{KE over normal with dim 2}, $\catX$ is a dominant resolving subcategory with $\inf_{X\in\catX}\{\depth X\} \ge 2$.
We then conclude that $\catX=\cm R$.
\end{proof}

\begin{rmk}\label{rmk:classification of dom resol}
The dominant resolving subcategories of $\catmod R$ were classified 
via certain functions on $\Spec R$
for a Cohen-Macaulay ring $R$ \cite[Theorem 1.4]{DT}
and for a general ring $R$ with some assumption \cite[Theorem 6.9]{Takahashi2}.
Therefore, \cref{prp:classify KE when two-dim normal local domain} gives
a complete classification of the KE-closed subcategories of $\catmod R$
for a two-dimensional noetherian normal ring $R$.
\end{rmk}


\section{Classifying several subcategories of a torsion-free class}
In this section, we classify several subcategories of a torsion-free class in $\catmod R$.
To clarify the discussion,
we often consider
an abelian category $\catA$ such that $\serre \catA = \tors \catA$.
Such abelian categories include the following examples:
\begin{itemize}
\item
$\catmod R$ when $R$ is a commutative noetherian ring (cf.\ \cite[Corollary 7.1]{SW}).
\item
The category $\catmod \LL$ of finitely generated $\LL$-modules 
when $(R,\LL)$ is a Noether algebra such that 
$\LL_{\pp}$ is Morita equivalent to a (nonzero) local ring for each $\pp \in \Spec R$ 
(cf.\ \cite[Corollary 3.24]{IK}).
\item
The category $\coh X$ of coherent $\shO_X$-modules when $X$ is a quasi-affine noetherian scheme 
(cf.\ \cite[Corollary 5.4]{Sai2}).
\end{itemize}
First of all, we explain that some classes of subcategories coincide in such an abelian category.
\begin{prp}\label{prp:subcat when Serre=tors}
Let $\catA$ be an abelian category such that $\serre \catA = \tors \catA$.
Then the following hold:
\begin{enua}
\item
$\tors \catA = \ce \catA$
(or equivalently, $\serre \catA = \tors \catA = \wide \catA = \ice \catA = \ce \catA $).
\item
$\torf \catA = \ike \catA = \ie \catA$.
\end{enua}
\end{prp}
\begin{proof}
(1)
It suffices to show that a CE-closed subcategory $\catX$ of $\catA$
is a torsion class of $\catA$.
There exists a torsion class $\catT$ of $\catA$ such that $\catX$ is a torsion class of $\catT$
by \cref{prp:CE=tors in tors}.
Then $\catT$ is a Serre subcategory of $\catA$ since $\serre \catA = \tors \catA$.
Thus, $\catX$ is also a torsion class of $\catA$ by \cref{prp:torf in Serre}.

(2)
We omit the proof since the same discussion as in \cite[Section 3]{Eno} works in this setting.
\end{proof}


Next, let us classify several subcategories of a torsion-free class.
(cf.\ \cref{dfn:subcat in ex cat}).
\begin{prp}\label{prp:tors=Serre in torf}
Let $\catA$ be an abelian category such that $\serre \catA = \tors \catA$,
and let $\catX$ be a torsion-free class of $\catA$.
\begin{enua}
\item
$\tors\catX = \serre \catX$ holds.
\item
If $\catX$ has an exact duality $\bbD$ (see \cref{dfn:ex dual}),
then $\tors\catX = \serre \catX = \torf \catX$ hold.
\end{enua}
\end{prp}
\begin{proof}
(1)
It is enough to show that a torsion class $\catS$ of $\catX$ is a Serre subcategory of $\catX$.
Since $\catS$ is a torsion class of a torsion-free class of $\catA$,
it is an IE-closed subcategory of $\catA$ by \cref{prp:char IE}.
Thus, it is a torsion-free class of $\catA$ by \cref{prp:subcat when Serre=tors}.
Especially, it is both a torsion class and a torsion-free class of $\catX$,
and hence $\catS$ is a Serre subcategory of $\catX$.

(2)
It is enough to show that 
a torsion-free class $\catS$ of $\catX$ is a Serre subcategory of $\catX$ by (1).
We can easily see that $\bbD(\catS)$ is a torsion class of $\catX$,
where $\bbD(\catS)$ is the subcategory of $\catX$ consisting of $\bbD A$ for $A\in \catS$.
Then $\bbD(\catS)$ is a Serre subcategory of $\catX$ by (1).
Therefore, $\catS$ is also a Serre subcategory of $\catX$ since $\bbD$ is an exact duality.
\end{proof}

This observation gives
a much simpler proof of \cref{KE=torf in 1-dim}
when $R$ is a Cohen-Macaulay ring with a canonical module $\omega$ such that $\dim R \le 1$.
For this, let us recall the classification of Serre subcategories 
of a torsion-free class of $\catmod R$.
\begin{fct}[{cf.\ \cite[Corollary 5.10]{Sai2}}]\label{fct:classify Serre in torf}
Let $R$ be a commutative noetherian ring and $\Phi$ a subset of $\Spec R$.
Then the assignments $\catS \mapsto \Ass\catS$ and $Z \mapsto \catmod^{\ass}_{Z} R$ give rise to 
mutually inverse bijections between 
$\serre(\catmod^{\ass}_{\Phi} R)$ and the set $\upper(\Phi)$ of upper sets of $\Phi$.
\end{fct}

\begin{cor}\label{prp:KE=torf when 1-dim CM with can mod}
For a Cohen-Macaulay ring $R$ with a canonical module $\omega$ such that $\dim R \le 1$,
we have $\ke(\catmod R)=\torf(\catmod R)$.
\end{cor}
\begin{proof}
The following equalities hold for any $\pp \in \Spec R$ by \cref{fct:classify Serre in torf}:
\[
\torf(\catmod^{\ass}_{\{\pp\}}R) = \{0\} \cup \ke(R;\{\pp\}),\quad
\serre(\catmod^{\ass}_{\{\pp\}} R)=\{0\} \cup \{\catmod^{\ass}_{\{\pp\}} R\}.
\]
Thus, it is enough to show that
$\torf(\catmod_{\{ \pp \}}^{\ass} R)=\serre(\catmod_{\{ \pp \}}^{\ass} R)$
for any $\pp \in \Min R$ by \cref{prp:char KE = torf}.
Consider the canonical duality $(-)^{\vee}:=\Hom_R(-,\omega) \colon (\cm R)^{\op} \simto \cm R$.
It restricts to an exact duality 
$\left(\catmod^{\ass}_{\Phi} R\right)^{\op} \simto \catmod^{\ass}_{\Phi} R$
for any $\Phi \subseteq \Min R$.
Indeed, for any $M\in \catmod^{\ass}_{\Phi} R$, we have that
\[
\Ass M^{\vee}=\Ass \Hom_R(M,\omega) = \Supp M \cap \Ass \omega
\subseteq \Supp M \cap \Min R \subseteq \Min M = \Ass M \subseteq \Phi,
\]
and hence $M^{\vee}\in \catmod^{\ass}_{\Phi} R$.
Thus, we have
$\torf\left(\catmod^{\ass}_{\Phi} R\right)=\serre\left(\catmod^{\ass}_{\Phi} R\right)$ 
by \cref{prp:tors=Serre in torf}.
In particular, we have that
$\torf(\catmod_{\{ \pp \}}^{\ass} R)=\serre(\catmod_{\{ \pp \}}^{\ass} R)$
for any $\pp \in \Min R$.
This finishes the proof.
\end{proof}

Conversely, the results of \S \ref{s:classify KE} yield
a classification of the torsion-free classes of a torsion-free class.
\begin{prp}\label{prp:classify torf in torf}
Let $R$ be a commutative noetherian ring with $\dim R \le 1$
and $\Phi$ a subset of $\Spec R$.
Then the assignments $\catF \mapsto \Ass\catF$ and $\Psi \mapsto \catmod^{\ass}_{\Psi} R$
give rise to mutually inverse bijections between 
$\torf(\catmod^{\ass}_{\Phi} R)$ and the power set $\pow(\Phi)$ of $\Phi$.
\end{prp}
\begin{proof}
Let $\catF$ be a torsion-free class of $\catmod^{\ass}_{\Phi} R$.
Then it is a KE-closed subcategory of $\catmod R$ by \cref{prp:KE=torf in torf}.
It is also a torsion-free class of $\catmod R$ such that $\Ass\catX \subseteq \Phi$ 
by \cref{KE=torf in 1-dim}.
Conversely, for any subset $\Psi$ of $\Phi$,
the subcategory $\catmod^{\ass}_{\Psi} R$ is clearly 
a torsion-free class of $\catmod^{\ass}_{\Phi} R$.
Therefore, the bijection in \cref{fct:Takahashi} restricts to 
the bijection between $\torf(\catmod^{\ass}_{\Phi} R)$ and $\pow(\Phi)$.
\end{proof}

Combining \cref{prp:tors=Serre in torf,fct:classify Serre in torf,prp:classify torf in torf},
we obtain the following commutative diagram 
for a commutative noetherian ring $R$ with $\dim R \le 1$:
\[
\begin{tikzcd}[row sep=10pt,ampersand replacement=\&]
\tors(\catmod^{\ass}_{\Phi} R) \arr{d,equal} \& \\
\serre(\catmod^{\ass}_{\Phi} R) \arr{d,phantom,"\subseteq",sloped} \arr{r,"\Ass"',"\iso"} \& \upper(\Phi) \arr{d,phantom,"\subseteq",sloped}\\
\torf(\catmod^{\ass}_{\Phi} R) \arr{r,"\Ass"',"\iso"} \& \pow(\Phi).
\end{tikzcd}
\]

As an immediate consequence of this,
we obtain the classifications of torsion(-free) classes of $\cm R$
when $R$ is a Cohen-Macaulay ring of dimension less than one.
\begin{cor}\label{prp:tors in cm}
Let $R$ be a Cohen-Macaulay ring with $\dim R \le 1$.
\begin{itemize}
\item 
$\tors(\cm R) = \serre (\cm R) = \torf(\cm R)$ holds.
\item
There is a bijection between $\torf(\cm R)$ and $\pow(\Min R)$.
\end{itemize}
\end{cor}
\begin{proof}
The category $\cm R$ is a torsion-free class of $\catmod R$ such that $\Ass(\cm R)=\Min R$ 
by \cref{ex:cm R is torf}.
Thus, the corollary follows from
\cref{prp:tors=Serre in torf,fct:classify Serre in torf,prp:classify torf in torf}
since $\pow(\Min R)=\upper(\Min R)$.
\end{proof}


\end{document}